\documentclass[a4paper,reqno,12pt,final]{amsart}
\usepackage{latexsym,amssymb,lastpage,amsfonts,euscript}
\usepackage{graphicx}
\usepackage{times,mathptmx,bm,amsmath,amsthm}
\usepackage{dcolumn}
\usepackage{showkeys}
\usepackage{placeins}
\usepackage{overpic}
\usepackage[bookmarks=false]{hyperref}

\usepackage[cp1251]{inputenc}
\usepackage[english,russian,ukrainian]{babel}

\newtheorem{lemma}{Lemma}
\newtheorem{theorem}{Theorem}
\newtheorem{example}{Example}
\newtheorem{corollary}{Corollary}
\newtheorem{definition}{Definition}

\graphicspath{{./pic/}}

\def\R{\mathbb R}
\def\C{\mathbb C}
\def\N{\mathbb N}

\def\Z{\mathbb Z}

\newcommand{\suml}{\sum\limits}
\newcommand{\intl}{\int\limits}

\DeclareMathOperator{\Arg}{Arg}
\DeclareMathOperator{\Ker}{Ker}

\usepackage{color}
\definecolor{gray10}{gray}{0.9}
\definecolor{gray20}{gray}{0.8}
\definecolor{gray30}{gray}{0.7}
\definecolor{gray40}{gray}{0.6}
\definecolor{gray50}{gray}{0.6}
\definecolor{gray60}{gray}{0.4}
\definecolor{gray70}{gray}{0.3}
\definecolor{gray80}{gray}{0.2}
\newcommand{\cbox}[1]{\colorbox{#1}{\textcolor{#1}{X}}}

\begin{document}
\selectlanguage{english}
\author{D. Sytnyk, V. Makarov, V. Vasylyk}
\email{sytnik@imath.kiev.ua}
\urladdr{www.imath.kiev.ua/$\sim$sytnik} 
\address{Institute of mathematics, National Academy of Sciences, Ukraine \\ 01601 Ukraine, Kiev-4,
3, Tereschenkivska st.} 
\title{Existence of the solution to a nonlocal-in-time evolutional problem.}

\begin{abstract}

This work is devoted to the study of a nonlocal-in-time evolutional problem for the first order differential equation in Banach space.
Our primary approach, although stems from  the convenient technique based on the reduction of a nonlocal problem to its classical initial value analogue,  uses more advanced analysis.
That is a validation of the correctness in definition of the general solution representation via the Dunford-Cauchy formula.
Such approach allows us to reduce the given existence problem to the problem of locating zeros of a certain entire function.
It results in the necessary and sufficient conditions for the existence of a generalized (mild) solution to the given nonlocal problem. 
Aside of that we also present new sufficient conditions which in the majority of cases generalize existing results. 


\end{abstract}
\maketitle

\section{Introduction}
The Cauchy problem for differential equation with a nonlocal in time condition:
\begin{equation}\label{bp}
      \begin{array}{l}
    u'_t+Au=f(t), \quad t \in [0,T]
    \\[1em]
    u(0)+g(t_1,t_2,\ldots, t_k, u) =u_0, \quad 0\leq t_0< t_1<\ldots < t_k \leq T,	
  \end{array}
\end{equation}
with $A$ being a densely defined strongly positive operator (the detailed definition will be given below) and $g:X\rightarrow D(A) \subseteq X$, is one of the important topics in the theory of differential equation theory and its application.
Interest in such problems stems mainly from the better effect of the nonlocal initial condition than the usual one in treating physical problems. 
Actually, the nonlocal initial condition from (\ref{bp}) models many interesting nature phenomena \cite{Allegretto1999}, \cite{Christensen1982}, where the normal initial condition $u(0) = u_0$ may not fit in. In addition some models of the control theory \cite{nonlocal_time_numerVabishchevich1982} and economical management problems may be represented in form (\ref{bp}) as well.

Even in a much simpler form:
\begin{subequations}\label{bp1}
    \begin{equation}\label{bp1:eq}
    u'_t+Au=f(t), \quad t \in [0,T]
    \end{equation}
    \begin{equation}\label{bp1:nc}
    u(0)+\suml_{k=1}^{n}\alpha_k u(t_k) =u_0,\quad 0<t_1<t_2<\ldots < t_n\leq T,
    \end{equation}
\end{subequations}
the model reflects an important case when one is more likely to have  more information about the solution at the times $t = t_1,t_2, \ldots, t_n$ rather than as in the classical case just at $t=0$.
Such situation is common for physical systems where an observer was not able to witness (or measure ) characteristics of the system at the initial time.
It is also intrinsic for modelling certain physical measurements performed repeatedly by the devices having relaxation time comparable to the delay between the measurements.
A simple example of this kind is a recovery of movement captured by a multi-exposure camera (streak-camera).

The particular cases of (\ref{bp1}) covers many well-known physical phenomena such as: problems with periodic conditions
$u(0)-u(t_1)=0$,  problems with Bicadze-Samarskii conditions $u(0)+\alpha_1 u(t_1)=\alpha_2 u(t_2)$, regularized backward problems  etc.
One can not underestimate the role of nonlocal problems having the form \eqref{bp1} in the theory of ill--posed problems, where they appear as natural counterparts of the improperly posed problem in the course of quasi-regularization technique \cite{IllPosedTikhonov1995}.  

Historically the first, to the authors best knowledge, work devoted to the problems with nonlocal-in-time conditions was the work of Dezin \cite{NonlocaAbstractDezin1965}.
In this work author studied the restrictions of abstract differential operators $u'(t)+A$ in a Hilbert space imposed by the non-local condition \eqref{bp1} ( see also \cite{bookDezin1980} and the references therein).
Similar theoretical technique in conjunction with a numerical scheme was used to solve nonlocal-in-space boundary value problem for elliptic partial-differential operators in the work of  Bicadze and Samarskii \cite{NonlocalBitsadze1969}.
Next addition to the theory was made by Gordeziani with co-authors \cite{nonloc_BS_Gordeziani1970}, \cite{nonloc_BS_Gordeziani1972}, 
They  rather generalized the previous results and proposed iterative problem-solving methods. Eigenvalue problems for elliptic operators with nonlocal conditions were considered in \cite{Bandyrskii2006} see also \cite{Sapagovas2012} for review of recent works in that direction.
In \cite{Byszewski1991}, \cite{Byszewski1992}, \cite{NonlocalAbsNonLinJackson1993}  authors stated the sufficient conditions for the existence and uniqueness of a solution to \eqref{bp} in a Banach space.
Initial analysis of \eqref{bp1} used in the present work was performed in \cite{Gavrilyuk2010}. Starting from the conditions similar to those proposed in \cite{Byszewski1992} authors of \cite{Gavrilyuk2010} develop  efficient parallel numerical methods for \eqref{bp1}. They use a rather general methodology developed for classical Cauchy problem for equation \eqref{bp1:eq}
Majority of later theoretical works are mainly devoted to the generalization of results, received in \cite{Byszewski1991}, to more wider classes of equations. 
 
Aside of that, some efforts have been made in the direction of sharpening the sufficient existence and uniqueness conditions for various specific classes of a nonlocal condition from \eqref{bp} and an operator coefficient $A$ \cite{Liang2002}. 

Most of previous research concerning the nonlocal problem \eqref{bp1}  were done under a strict constraint on quantities  $\alpha_i$, from the nonlocal condition \cite{Byszewski1991}, \cite{NonlocalAbsNonLinJackson1993}, \cite{Byszewski1992}:
$$
\suml_{i=1}^{n} |\alpha_i| \leq 1,
$$
allowing the transition from \eqref{bp1} to the classical Cauchy problem.
This inequality obviously gives only a sufficient condition.
Indeed, if one had a solution of ordinary Cauchy problem for equation from (\ref{bp}) he could simply choose such a nonlocal condition that mentioned constrain fails but the solution exist and unique.   

In \cite{NonlocalAbsNonLinNtouyas1997} authors, using spectral characteristics of $A$, obtained  new slightly weaker constraint
\begin{equation}\label{estLiang2002} 
\suml_{i=1}^{n} |\alpha_i|e^{-\rho t_i} \leq 1,
\end{equation}
here $\rho$ has the same meaning as in (\ref{spectrA}). 

Here we will introduce more advanced and yet quite straightforward technique for the treatment of nonlocal problems.
To do that, in section \ref{sec1} we reduce the given problem \eqref{bp1} to the corresponding Cauchy problem  with ordinary initial value condition as in \cite{Byszewski1991}, \cite{NonlocalAbsNonLinNtouyas1997}. 
The subsequent analysis, presented in section \ref{sec2}, study the correctness of the solution operator defined via the Dunford-Cauchy formula on a feasible integration contour.
It is equivalent to the calculation of zeros set for a certain entire function or the polynomial corresponding to it.  
The resulting necessary and sufficient conditions for the existence and uniqueness of the solution take in to account both the nonlocal condition and the spectral information of $A$.
Section \ref{sec3} is devoted to the situation when zeros of the mentioned entire function can not be calculated directly.
That may be caused by a dependence of the nonlocal condition on some additional parameter like in applications of quasi-regularization technique, or because of a large number of time moments $t_k$ in \eqref{bp1}, etc.
Such situation as it will be shown,  can be circumvented using some available zeros estimates \cite{Milovanovic2000}, \cite{Milovanovic2000a} for the mentioned entire function in its reduced to the polynomial form. It results in the sufficient conditions for \eqref{bp1} which nonetheless outperform \eqref{estLiang2002} in many important applications.

\section{Reduction of nonlocal problem to classical Cauchy problem.}  \label{sec1}
From now on we assume that the coefficient  $A$ from \eqref{bp}  is a densely defined strongly positive operator acting in Banach space 
$X \supseteq D(A) \rightarrow X$ \cite{bHilPhil1997}, \cite{bGavrilyuk2011}. 
Its spectrum $\Sigma (A)$ lies in a sector 
\begin{equation}\label{spectrA}
\Sigma = \left\{
z=\rho+r e^{i\varphi}:\quad r \in [0,\infty), \ \left|\varphi\right| \leq \theta < \frac{\pi}{2}, \rho > 0
\right\},
\end{equation}
while the resolvent $R\left(z,A\right)$ of $A$ satisfies the following estimate on the boundary of spectrum $\Gamma_\Sigma$  and outside of it 
\begin{equation}\label{estrez}
\left\|(zI-A)^{-1}\right\|\leq \frac{M}{1+\left|z\right|}, 
\end{equation}
with $\left\| \cdot \right\|$ being an operator norm.


The class of strongly positive operators plays an important role in  applications of functional analysis to the theory of partial differential equations, dynamical systems, numerical analysis, etc.  Strongly--elliptic partial differential operator defined on a bounded Lipschitz domain is a strongly positive operator with spectral parameters that can be estimated from the coefficients of elliptic operator \cite{fujita}, similar is true for a general elliptic pseudo-differential operator. 

%
%
%
The theory of sectorial operators in its present form was developed in the works of Hille, Dunford, Philipps  \cite{b_cp_Fattorini1984}. 
According to this theory every closed strongly-positive operator generates a one parameter semi-group 
$T(t)= e^{-t A}$ 
which acts as a propagator for a solution to \eqref{bp1:eq}.
That is any solution $u(t) \in D(A)$ of the differential equation \eqref{bp1:eq} has the following representation 
\begin{equation}\label{bp1int}
  u(t)=e^{-At}u(0)+\intl_0^t{e^{-A(t-\tau)}f(\tau)}d\tau.
\end{equation}
Recall that the opposite is not true in general \cite{b_cp_Fattorini1984}, since even if $u(t)$ satisfies \eqref{bp1int} it does not need to belong to the domain of $A$. It is also well know that $D(A)$ is dense in X so there always exists a sequence of elements from $D(A)$ converging to $u(t)\in X$ defined by \eqref{bp1int}. 
A function $u(t)$ satisfying \eqref{bp1int} is called a \textit{generalized }\cite[p. 30]{b_cp_Fattorini1984} or sometimes \textit{mild } \cite[p. 117]{arendt} \textit{solution} to problem \eqref{bp1}. Formula \eqref{bp1int} becomes more convenient than the original equation \eqref{bp1:eq} in cases where the classical Cauchy problem for \eqref{bp1:eq} with a given $u(0)$ is considered. In our case of problem \eqref{bp1} that distinction between \eqref{bp1int} and \eqref{bp1:eq} is not so obvious because $u(0)$ is not given directly.

Here we intend to derive a direct representation of the solution to \eqref{bp1}. 
Let us start from representation \eqref{bp1int} and assume the existence of an initial value $u(0)$ such that  $u(t)$ defined by \eqref{bp1int} satisfies nonlocal condition \eqref{bp1:nc} (the precise conditions for the existence of such $u(0)$ will be stated below).

By substituting  the formula for $u(0)$ from \eqref{bp1:nc} into representation \eqref{bp1int} and evaluating the result at  $t_i, i=\overline{1,n}$ we obtain the system of equations 
\begin{equation}
\label{bp1LS}
  \begin{array}{l}
    u(t_i)=e^{-At_i}\left[u_0-\suml_{k=1}^{n}\alpha_k u(t_k) \right]+\intl_0^{t_i}{e^{-A(t_i-\tau)}f(\tau)}d\tau,\\
    i=\overline{1,n}.
  \end{array}
\end{equation}

Next we multiply each part of $i$-th equation by $\alpha_i$ and then sum up the resulting equalities. It gives us the following: 

\begin{equation}
  \label{sum_t_i}
  \begin{split}
  \suml_{i=1}^n \alpha_i u(t_i) =& \suml_{i=1}^n \alpha_i e^{-At_i} u_0-\suml_{i=1}^n \alpha_i e^{-At_i}
  \suml_{k=1}^n \alpha_k u(t_k)+ 
  \\
  &+\suml_{i=1}^n \alpha_i \intl_0^{t_i} e^{-A (t_i-\tau)}f(\tau) d\tau.
    \end{split}
\end{equation}
 Now let us put $\suml_{i=1}^n \alpha_i u(t_i)=w$, then  (\ref{sum_t_i}) can be represented as follows:
 $$
 w = {-\suml_{i=1}^n \alpha_i e^{-At_i} w} +\suml_{i=1}^n \alpha_i e^{-At_i} u_0+ \suml_{i=1}^n \alpha_i \intl_0^{t_i} e^{-A (t_i-\tau)}f(\tau) d\tau,
 $$
 the last equality can be regarded as an operator equation with respect to $w$. Using the notation 
 $B(A)=I +\suml_{i=1}^n \alpha_i e^{-At_i}$, we rewrite it to get
\begin{equation}\label{eqnoneq1}
B w = B u_0 -u_0 +\suml_{i=1}^n \alpha_i \intl_0^{t_i} e^{-A (t_i-\tau)}f(\tau) d\tau.
\end{equation} 
Here unknown $w$ appears only on the left-hand side of \eqref{eqnoneq1}, hence in order  to solve \eqref{eqnoneq1} for $w$ we need to assume the existence and boundedness of the operator valued function $B^{-1}(A)$ inverse to $B(A)$ which we will call reduction operator.

As it will turn out later (see Theorem \ref{thmNCExist}) such assumptions with respect to $B^{-1}(A)$ is the only thing we need to reduce the question about the existence of solution to \eqref{bp1} to the question about the existence of corresponding solution to a classical Cauchy problem associated with \eqref{bp1:eq}.
For the time being let us assume that $B^{-1}(A)$ is a properly defined bounded operator valued function, in that case  
\[ 
 w = u_0- B^{-1} u_0 + B^{-1}\suml_{i=1}^n \alpha_i \intl_0^{t_i} e^{-A (t_i-\tau)}f(\tau) d\tau.
 \]
Our last step is to substitute $u(0) = u_0-w$ \eqref{bp1int}, and finally get a direct representation of solution to \eqref{bp1} free of the unknown values $u(t_k)$:
 \begin{equation}
   \label{bp1IntRed}
    \begin{split}
   u(t)=&e^{-At}\left[B^{-1} u_0 - B^{-1}\suml_{i=1}^n \alpha_i \intl_0^{t_i} e^{-A (t_i-\tau)}f(\tau) d\tau\right] +
   \\
   &+\intl_0^t{e^{-A(t-\tau)}f(\tau)}d\tau.
    \end{split}
 \end{equation}
 
For further analysis of relationship between the solution to \eqref{bp1} and the initial data $\alpha_i, t_i$ we will need a following definition from the operator function calculus \cite[p. 167]{bHilPhil1997},  \cite{Clem}.
\begin{definition}\label{thmDunford}
Let $f(z)$ be a complex valued function  analytic in the neighbourhood of the spectrum $\Sigma(A) \subset \C$ and at the infinity. Suppose that there exist an open set $V \supset \Sigma(A)$ with the boundary $\Gamma$ consisting of a finite number of rectifying Jourdan curves such that $f(z)$ is analytic in $V\cap \Gamma$,  then $f(A)$ can be defined as follows 

\begin{equation}
		\label{reprDunford}
		f(A) x=f(\infty)I+\frac{1}{2\pi i} \intl_{\Gamma}f(z) R(z,A)x dz,
\end{equation}
here we assume that the integral is taken over the positively oriented contour $\Gamma$ which may pass trough $\infty$.
\end{definition}
The integral in \eqref{reprDunford} is called a Dunford-Cauchy integral.

Representation \eqref{reprDunford} once applied to $B^{-1}(A)$ lead us to the formula
\begin{equation}
		\label{reprDelta}
		B^{-1}(A) u=I+\frac{1}{2\pi i}\intl_{\Gamma_A} \frac{1}{1+\sum_{k=1}^n{\alpha_k e^{(-t_k z)}}} R(z,A) u dz,
\end{equation}
from which it is clear that the only possible source of singularities of reduction operator $B^{-1}(z)$ would be a set of zeros of denominator from \eqref{reprDelta}. 
Thus the function $B^{-1}(A)u$ is properly defined in the sense of Definition \ref{thmDunford} if and only if all the zeros of 
\begin{equation}
	  \label{zerosExp}
	 B(z)=1+\sum_{k=1}^n{\alpha_k e^{(-t_k z)}},\quad z \in \mathbb{C}_+,
\end{equation}
belong to a set  $\mathbb{C} \backslash \Sigma $. 
Now we can formalize our previous analysis as a theorem. 
\begin{theorem}\label{thmNCExist}
Let  $A$ be a strongly positive linear operator with the spectral parameters $(\rho, \theta)$, and  $f(t) \in L^1((0;T),X)$ be a given function. Then the generalized solution \eqref{bp1IntRed} exists if and only if the set of zeros $ \Ker(B(z))\equiv\left\{z:B(z)=0, z \in C \right\}$, of $B(z)$ associated with nonlocal condition \eqref{bp1:nc}, satisfy the inclusion 
\begin{equation}\label{zerosB}
\Ker(B(z))\subset \mathbb{C} \backslash \Sigma. 
\end{equation} 
\end{theorem}
\begin{proof}
\textit{1. Necessary conditions.}

We assume that the solution to \eqref{bp1} exists. 
Then this solution should satisfy formula \eqref{bp1int} (see \cite[ Proposition 3.1.16]{arendt}) which in the homogeneous case $f(t) \equiv 0$ is reduced to 
\begin{equation*}\label{exp_sol_repr}
u(t)=e^{-At}u(0).
\end{equation*}

Once the validity of \eqref{bp1int} (or \eqref{exp_sol_repr}) for $u(t)$ is established we can use it along with \eqref{bp1:nc} to get operator equation \eqref{eqnoneq1}, in a way described above.
This equation, as we already discovered,  has a nontrivial solution only if the inclusion \eqref{zerosB} is valid. 
\newline
\textit{ 2. Sufficient conditions.}
To prove the sufficiency let us focus our attention on function $u(t)$ defined by  \eqref{bp1IntRed}. 
Operator valued function $e^{-A (t-s)}$ from the integrands appearing in \eqref{bp1IntRed} is differentiable so the integrals  are convergent for any $f(t) \in L^1((0;T),X)$, and therefore $u(t)$ is properly defined and bounded once $B^{-1}(A)$ is properly defined. Next from the Dunford-Cauchy integral representation \eqref{reprDelta} we observe that $B^{-1}(A)x$  is properly defined and bounded $\forall x \in X$ as long as the propositions of the theorem regarding the spectrum of $A$ are fulfilled and \eqref{zerosB}  is valid.  
It remains to show that  $u(t)$ defined by  \eqref{bp1IntRed} is indeed a generalized solution to \eqref{bp1}. This can be easily done by substitution of the representation of $u(t)$ into \eqref{bp1:nc}. 
\end{proof}
This theorem answers the question regarding the existence of solution to problem \eqref{bp1}. Such result alone have a limited theoretical and practical utility.  Two supplementary questions regarding the uniqueness of $u(t)$  and its stability with respect to the initial data are need to be studied as well.

We discuss the stability first. In order to show that assumptions stated in Theorem \ref{thmNCExist} imply the continuous dependence of solution $u(t)$ defined by \eqref{bp1IntRed} on $u_0$ $f(t)$ and the parameters of nonlocal condition we recall similar results for classical Cauchy problem.
If the theorem's assumptions regarding $A$ are valid then the classical Cauchy problem associated with equation \eqref{bp1:eq} is \textit{well posed} in the sense of \cite[p. 29]{b_cp_Fattorini1984}. Its solution represented by \eqref{bp1int} exist for any initial state  $u(0) \in X$ and for $u(0) = u_0 - w$ in particular. The last expression continuously depends on the initial data. 

The uniqueness of the solution to \eqref{bp1}  can be justified in similar manner as for the ordinary abstract Cauchy problem associated with equation \eqref{bp1:eq}.
Clearly any function $v(t)$ satisfying \eqref{bp1:eq} can be represented in the form 
$$
v(t)=e^{-At}v(0).
$$
Assume that $u(t)$ is a solution to nonlocal problem \eqref{bp1}  $\forall t\geq 0$  $u(t) \in D(A)$, and consider a function
$$
w(t)=e^{-At}u(0).
$$  
Both $w(t)$, $u(t)$  satisfy differential equation \eqref{bp1:eq}, and so does their difference $p(t)= w(t)-u(t)$. 
The corresponding Cauchy problem is well-posed thus 
it has a unique solution for any initial state $u_0 \in D(A)$ \cite[Theorem 23.8.1]{bHilPhil1997}. 
We take $p(0)=0$ as such state. Then the difference $p(t)\equiv 0$ everywhere since it the only  solution to \eqref{bp1:eq} with the zero initial condition.

\begin{example}\label{ex1}
To demonstrate the application of Theorem \ref{thmNCExist} let us consider the following nonlocal problem 
\begin{equation}\label{exNCP1p}
  \begin{array}{l}
    u'_t+Au=f(t), \quad t \in [0,T]\\[1em]
    u(0)+\alpha_1 u(t_1) =u_0,\quad 0<t_1\leq T.
  \end{array}
\end{equation}
For such nonlocal condition  $B(z)$  will have the form 
$$
B(z) = 1 +\alpha_1e^{-zt_1}.
$$
Its representation permits us to write the set of zeros $\Ker{B(z)}$ in a closed form 
\begin{equation}
\begin{split}\label{kerB1p}
\Ker(B(z)) = & -\frac{1}{t_1}\ln\left(-\frac{1}{\alpha_1}\right)=
\\
= &-\frac{1}{t_1}\left[\ln\left|\frac{1}{\alpha_1}\right|+i\left( \Arg\left(-\frac{1}{\alpha_1}\right)+2\pi m\right) \right], \quad m\in \Z 
\end{split}
\end{equation}
here $\Arg\left(\cdot\right)$ stands for the principal value of argument.  
Assuming that the operator $A$  has the spectral parameters  $(\rho, \theta)$
$$
z=x+iy \in \mathbb{C} \backslash \Sigma \Leftrightarrow |y|>(x-\rho)\tan\theta.
$$
One can observe from the last inequality that if the principal value of logarithm satisfies condition \eqref{kerB1p} the same is true for the entire set of the logarithm values, so one can safely put $m=0$ in \eqref{kerB1p}. 

Condition \eqref{zerosB} for problem \eqref{exNCP1p}, therefore is equivalent to the following inequality 
\begin{equation}\label{estNC1p}
\left|\Arg\left(-\frac{1}{\alpha_1}\right)\right|>\left(\ln\left|{\alpha_1}\right|-t_1\rho\right)\tan\theta.
\end{equation}
Note that unlike \eqref{estLiang2002} or earlier estimates by Byszewski this inequality takes into account both spectral parameters of $A$. Closed form representation for $\Ker{B(z)}$ and the proposition of Theorem \ref{thmNCExist} guarantee that \eqref{estNC1p} forms a necessary and sufficient conditions for the existence of generalized solution to \eqref{exNCP1p}. 

Example \ref{ex1} is instructive in a sense that for such nonlocal problem  one can easily compare conditions \eqref{estLiang2002} and \eqref{estNC1p} graphically. 
This comparison is given on Figure \ref{nc1pFig1} where we depict three sets of admissible values of $\alpha_1 \in \C$ for $t_1=1$. 
\begin{figure}[ht]\label{nc1pFig1}
\vspace*{-3pt}
\centering\includegraphics[width=0.95\linewidth]{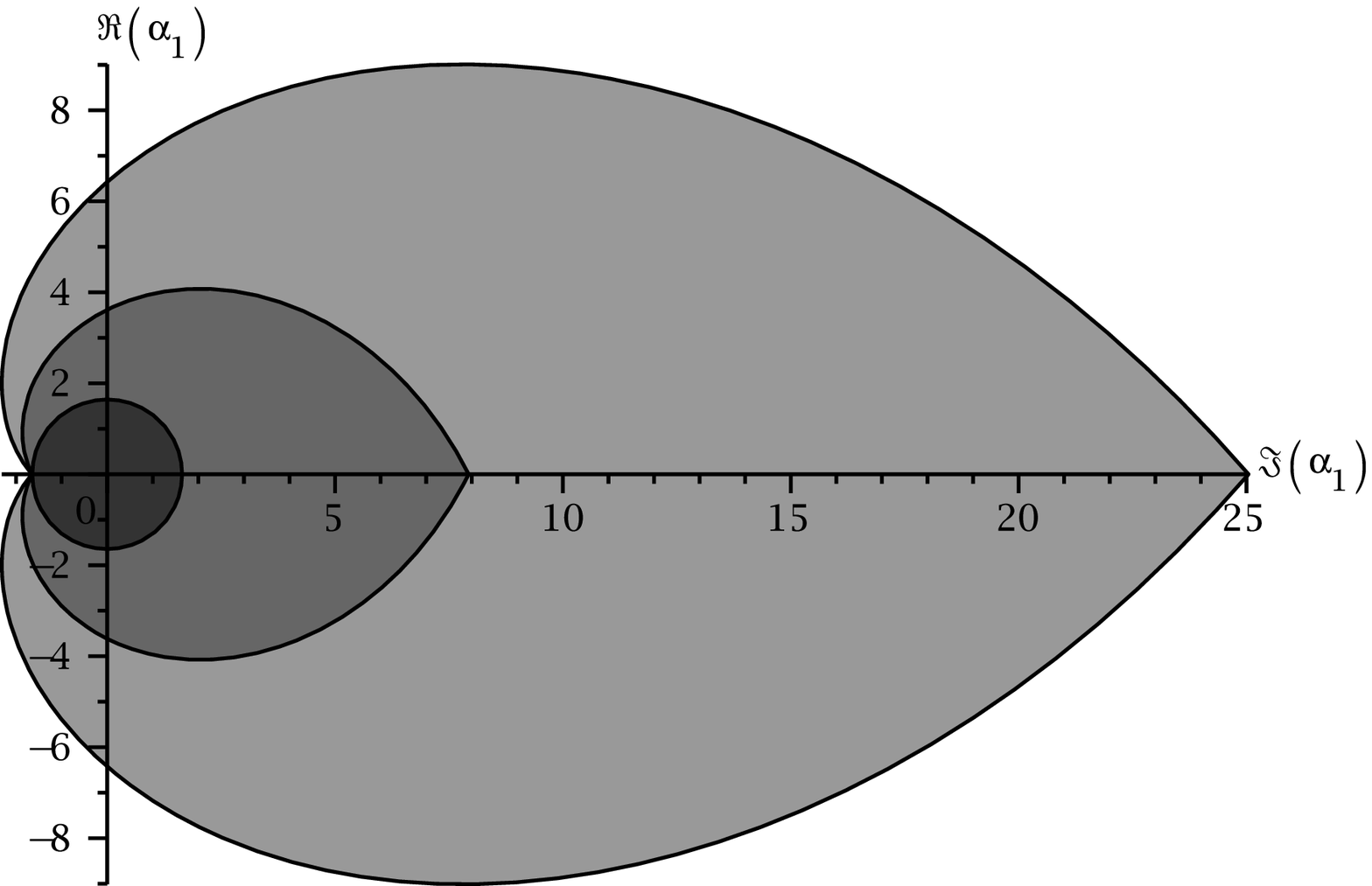}
\vspace*{-3pt}
\caption{Admissible values of parameter $\alpha_1 \in \C$ from nonlocal condition of problem  \eqref{exNCP1p} obtained by: \cbox{gray80} -- estimate  \eqref{estLiang2002} ($\theta=\pi/4$), \cbox{gray60} --  estimate \eqref{estNC1p} ($\theta=\pi/4$) and  \cbox{gray40} by estimate \eqref{estNC1p} ($\theta=\pi/6$). Spectral parameter $\rho=1$ everywhere.}
\end{figure}
Observe that for the operator $A$ with spectral parameters $(1,\pi/4)$ ($\rho=1$, $\theta=\pi/4$) the set of admissible $\alpha_1$ obtained from  \eqref{estNC1p} (interior of the region coloured in \cbox{gray60}) contains in itself as a subset the admissible set obtained by  \eqref{estLiang2002} (coloured in \cbox{gray80}). This set remains the same for the whole family of sectorial operator coefficients with some fixed $\rho$ and $\forall \theta \in [0, \pi/2]$ since \eqref{estLiang2002} are independent of $\theta$. 
While in reality the admissible set grows larger when we make $\theta$ smaller.  Check for example the corresponding set for the case  $\theta=\pi/6$ obtained using \eqref{estNC1p} which is coloured in \cbox{gray40} on Figure \ref{nc1pFig1}.  In the limiting case of $\theta=0$ when $A$ is self-adjoint this set becomes equal to $\C \backslash (-\infty; -e^{t_1\rho})$.  

As partial case of \eqref{exNCP1p} one can consider the problem
with $A = -\frac{d^2}{dx^2} + \pi^2 - 1,$ $D(A)=\left\{v(x) : v \in H^2, \ v(0) = v(1)=0 \right\}$, $\rho = 1$ and \eqref{bp1:nc} in the form $u(0)+2e u(1) = 3\sin(\pi x)$. 
Even thought the solution to such problem exists and has the form $u(x,t) = e^{-t} \sin(\pi x)$ condition \eqref{estLiang2002} fails. 
Meanwhile the approach described above shows that solution to the mentioned problem exists $\forall \alpha_1 \in \C \backslash (-\infty; -e)$. 
\end{example}

To convince the reader that situation complicates when the nonlocal condition consists of more than one value of unknown at the given times we study a problem with two-point nonlocal condition. 

\begin{example}\label{exBS}
Let us consider the problem 
	\begin{equation}\label{exNCP2p}
  \begin{array}{l}
    u'_t+Au=f(t), \quad t \in [0,T],\\[1em]
    u(0)+\alpha_1 u(t_1) + \alpha_2 u(t_2)=u_0,\quad 0<t_1<t_2\leq T.
  \end{array}
\end{equation}

Such nonlocal condition yields the following $B(z)$:  
$$
B(z) = 1 +\alpha_1e^{-zt_1}+\alpha_2e^{-zt_2},
$$
whence it is clear that a closed form representation of $\Ker{B(z)}$ is not available in general. 
The function $B(z)$ remains entire for any fixed $\alpha_1,\alpha_2$, $t_1,t_2$.
So its roots can be accurately approximated numerically using various methods \cite{Henrici1974} (Newtons method and its modifications, gradient methods,  numerical methods based on the argument principle and numerical quadratures, etc.) which implementations are available as a part of many modern mathematical programs (Octave, Maxima, Matlab, Maple).
  
By fixing the values
\begin{equation}\label{par1exNCP2p}
\alpha_1=-0.13,\  \alpha_2=3,\  t_1=1/2,\  t_2=1, 
\end{equation}
we  get
$$
B(z) = 1 -0.13 e^{-\frac{z}{2}}+3e^{-z}.
$$
If we additionally assume that operator $A$ has the spectral parameters  $(0, \theta)$ it becomes obvious that condition \eqref{estLiang2002} is inapplicable in such situation ($0.13+3 >1$).  
The approximate calculation of zeros of $B(z)$ carried by Maple package or, to be more exact, the function \verb#Analytic# (being the implementation of modified Newton's method) gives us  :
$$
\Ker(B(z)) = -2.09255541146 + 4\pi i,
$$ 
here all given digits are significant. 
Combining this information, Theorem \ref{thmNCExist} and the fact that the spectrum of $A$ lies in a right half-plane of $\C$ we conclude that the generalized solution to problem \eqref{exNCP2p} exists for any $\theta \in [0, \pi/2]$.
\end{example}
All in all, the performed numerical analysis will always allow us to clarify the existence of a solution to \eqref{bp1} as long as the nonlocal parameters from \eqref{bp1:nc} are fixed.
For many application of \eqref{bp} with $n>2$ this is not enough as one still would like to have some a priory information about the admissible parameters set rather than simply check the existence of solution for some fixed values  of nonlocal parameters. 
This often happens in applications to control theory where one must guarantee the solution's existence for a certain submanifold in the space of parameter values. 
In the remaining part of the work we propose the technique how to estimate $\Ker{B(z)}$ by means of some well known bounds on roots of polynomial.

\section{Zeros of $B(z)$ and equivalent problem for polynomial}\label{sec2}

At first we assume that all $t_k$ from nonlocal condition \eqref{bp1:nc} are rational numbers.
This assumption in itself is quite adequate in practice because the computer representation of $t_k$ rely on a fixed size mantissa \cite{ieeeKahan1987}.
Every  $t_k$ admits the representation 
\begin{equation*}
t_k= \frac{\lambda_k}{\mu_k}, \quad \lambda_k \in \mathbb{Z},\quad  \mu_k \in \mathbb{N},
\end{equation*}
Next we set $c_k=\frac{Q \lambda_k}{\mu_k}$, with  $Q=\mbox{LCM}(\mu_1,\mu_2, \ldots, \mu_n)$. The function 
$B(z)$ is periodic with period $2 \pi Q i$. Thus, using the arguments from Example \ref{exNCP1p} the set $\mathbb{C} \backslash \Sigma$ can be safely reduced to $D_{Q} \backslash  \Omega_Q$ 
$$
\Omega_Q=\Sigma(A) \cap D_{Q},
$$
here $D_{Q}$ is a strip around the real axis with the width $2\pi Q$ (See Figure \ref{picQSPhi}, a)).

Now we would like to make use of Theorem \ref{thmNCExist}. For that one needs to check whether $\Ker{B(z)} \subset D_{Q} \backslash  \Omega_Q$.  This problem is just as difficult as the corresponding problem for $\C\backslash \Sigma$.
%

A mapping  
\begin{equation}
\label{bp_transf1}
\varphi(z)=\exp(-z/Q)
\end{equation}
transforms  \eqref{zerosExp} into the following form 
\begin{equation}
		\label{zerosPol}
		P(z)=1+\suml_{k=1}^{n}\alpha_k z^{c_k}.
\end{equation}
It is well known \cite{Henrici1974}, that \eqref{zerosExp} is one-to-one conformal mapping of $\Omega_Q$  onto  $\Phi$ (see Figure \ref{picQSPhi} b)).
\begin{figure}[ht]%
\vspace*{-3pt}
\begin{tabular}{lr}
		\begin{overpic}[width=0.49\linewidth]%
		{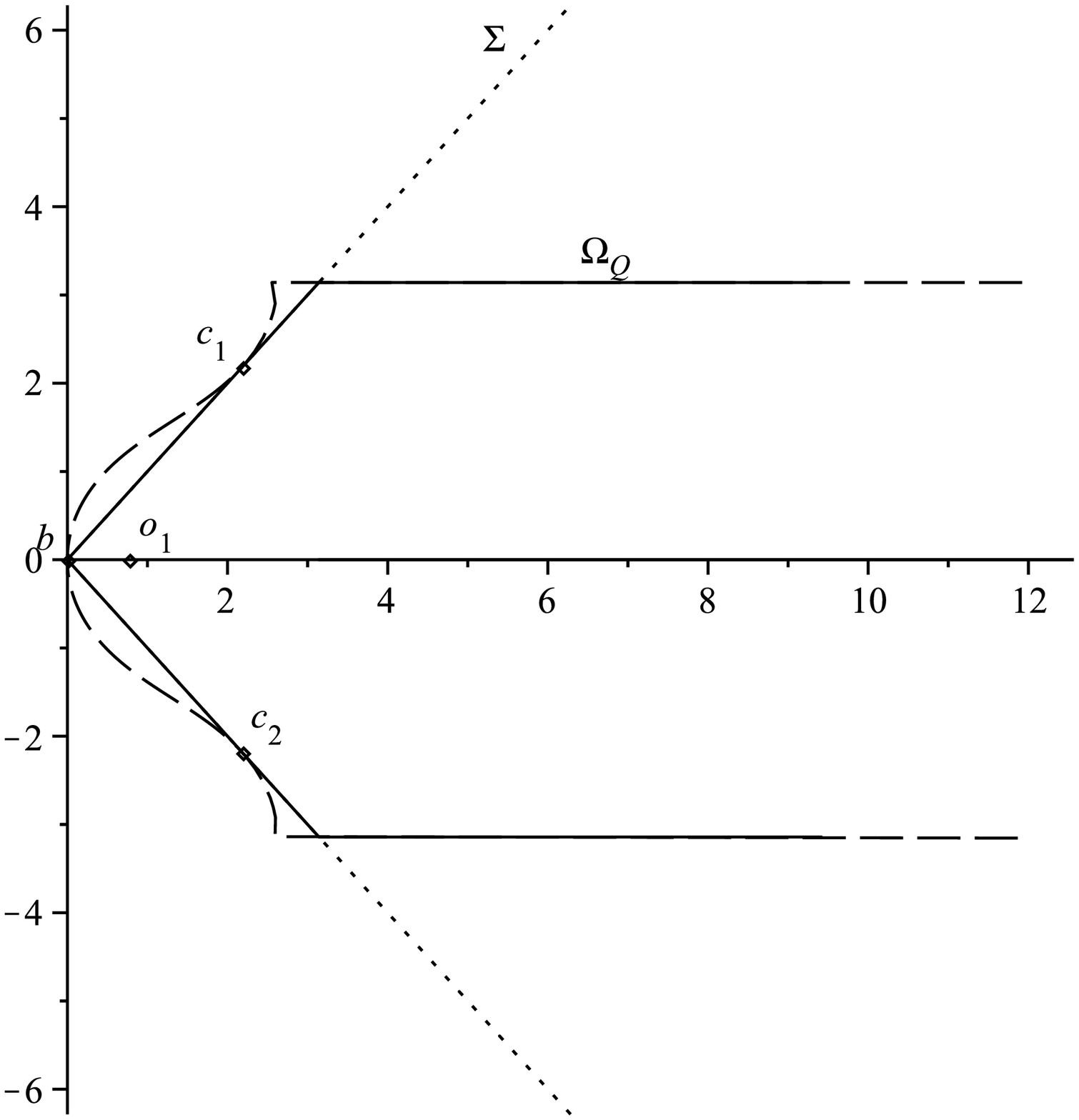}
			\put(15,95){\textbf{a)}}
		\end{overpic}&
		\begin{overpic}[width=0.49\linewidth]%
		{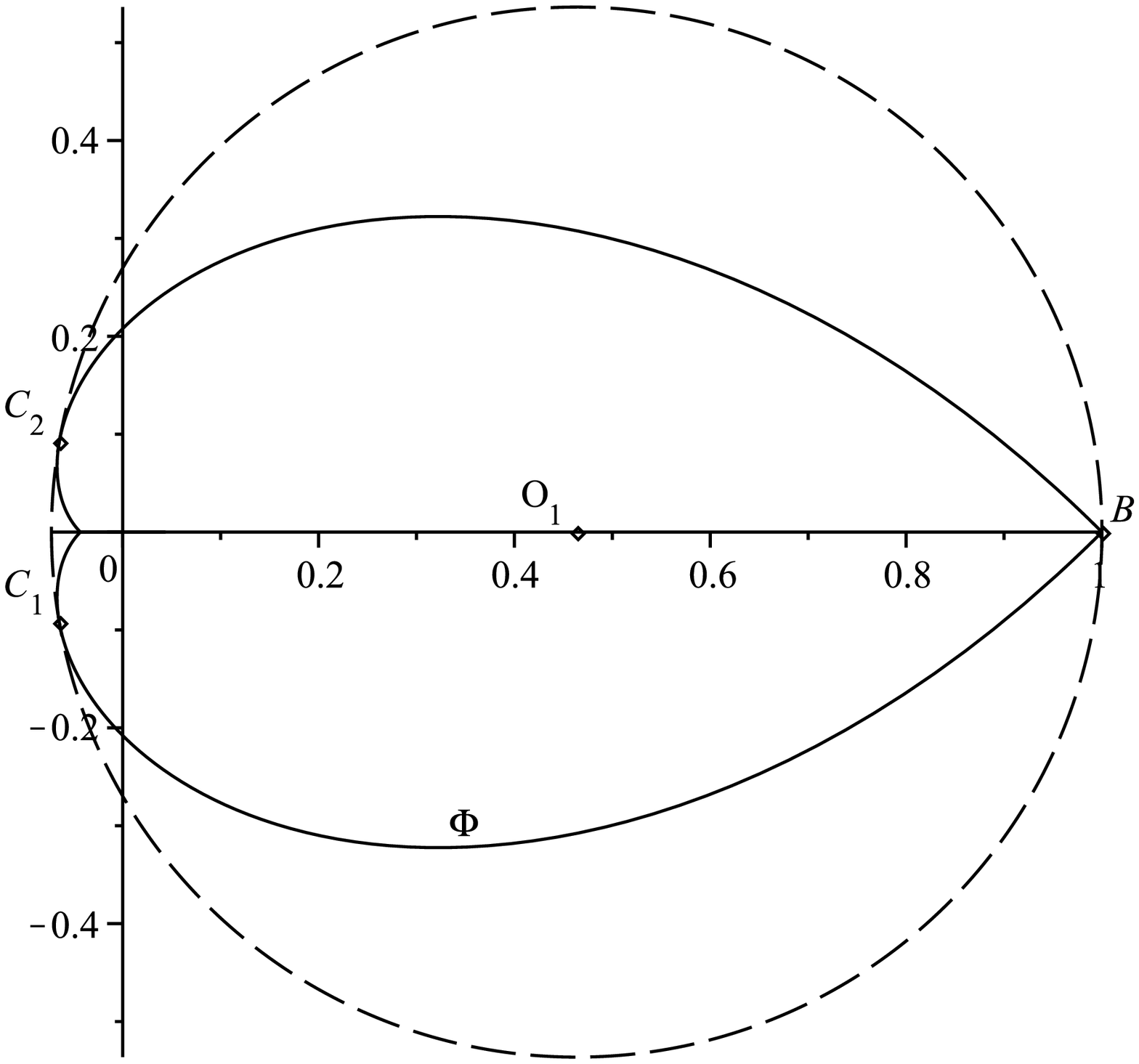}
			\put(15,95){\textbf{b)}}
		\end{overpic}
				\\
\end{tabular}
\vspace*{-3pt}
\caption{The regions $\Sigma$, $D_Q$ and their images for $A$ with spectral parameters $\theta=\pi/4$, $\rho=0$ ($Q=1$):  
 \textbf{a)} intersection of spectrum  $\Sigma$  and the set $D_Q$;
 \textbf{b)} the region $\Phi$ and its encompassing circle. Points $c_{1/2},b,o_1$ are the preimages of $C_{1/2},B,O_1$.}  
\label{picQSPhi}
\end{figure}
By using it we achieved two goals. First of all the selected mapping transforms the entire function $B(z)$ into polynomial $P(z)$ with real coefficients (provided that all $\alpha_i$ are real). Secondly this mapping reduces the zeros finding problem for the exterior of $\Sigma$ to the same problem for the exterior of a bounded set $\Phi \subset \C$. 
Or speaking more precisely the conditions guarantying that all roots of $P(z)$ lie outside $\Phi$ or equivalently $\Ker{B(z)} \subset D_{Q} \backslash  \Omega_Q$ would be necessary and sufficient to prove that existence and uniquenesses of solution \eqref{bp1IntRed}. 
The majority of results related to such conditions for polynomials are devoted to the situation when a circle is considered in place of  $\Phi$
(to review existent results in that field see \cite{Milovanovic2000a, Milovanovic2000}, as well as \cite{Henrici1974, Kravanja2000}). 

That is why we first encircle $\Phi$ and then use readily available zero-free conditions for that circle. Such approach will make the resulting conditions only sufficient for all $\theta \in [0,\pi/2)$ except for the limiting case $\theta = \pi/2$ when $\Phi$ is a circle by construction.  

For any given operator $A$ with spectral parameters $(\rho,\theta)$ the boundary of $\Phi$ can be parametrized as follows
$$
\partial \Phi = \left\{\exp\left(\frac{-Z(x)}{Q}\right): x\in[0,+\infty]\right\},
$$
where $Z(x)$ is a parametrization of $\partial\Omega_Q$:
$$
Z(x)=\rho+x+i \cdot\left\{
\begin{array}{ll}
	x \tan{\theta},&\quad x\tan\theta< Q\pi,\\
	Q\pi,&\quad x\tan\theta\geq Q\pi.\\
\end{array}
\right.
$$
A closer look at the expression for $\partial \Phi$ unveils that a vertical linear diameter of $\Phi$ is proportional to the magnitude of spectral angle, and the horizontal diameter of $\Phi$ is reversely proportional to $\rho$. 
This observation suggests us to describe the encompassing circle as a circumcircle of a triangle with the vertices 
$$
B=\max_{z \in \partial \Phi}\Re(z)+0i=\exp(-\rho/Q),
$$
and  $C_{1/2} \in \partial \Phi $ which are symmetric with respect to the real axis. 
The coordinates of $C_{1}$ are chosen to maximize the 
distance $|O_1-B|$ under the constrain $|O_1-B|^2=|O_1-C_i|^2$,  here $O_1$ is a circumcentre of $\triangle_{BC_1C_2}$. 
Using the definition of $Z(x)$, \eqref{bp_transf1} and some basic facts from calculus we reduce the mentioned maximization problem to the following equation
\begin{equation}\label{eq_maxdist}
\begin{split}
\exp\left(-\frac{2x}{Q}\right)\left[\cos\left(\frac{x\tan\theta}{Q}\right)-\tan(\theta)\sin\left(\frac{x\tan\theta}{Q}\right)\right]+
\\
 +\cos\left(\frac{x\tan\theta}{Q}\right)+\tan(\theta)\sin\left(\frac{x\tan\theta}{Q}\right)
=  2\exp\left(-\frac{x}{Q}\right).
\end{split}
\end{equation}
It has a positive solution for $Q\in \N$ and $\forall \theta \in [0,\pi/2]$. Assume that $x_d$ is a solution of \eqref{eq_maxdist}, then
$$
C_{1/2}=\varphi(\rho+x_d \pm i x_d \tan\theta),
\quad
O_1= 
 \frac{\varphi(2\rho)-\Re(C_1)^2-\Im(C_1)^2}{2\left(\varphi(\rho)-\Re(C_1)\right)},
$$
while the radius of circumcircle $r=\varphi(\rho)-O_1$ (The picture of $\Phi$, its encompassing circle along with their inverse images  are shown in Figure \ref{picQSPhi}  ).


\section{Sufficient conditions for existence of solution}\label{sec3}

Let us review what we have done so far. Starting from Theorem \ref{thmNCExist} we reduce the problem of clarifying whether $\Ker{B(z)} \subset \C \backslash \Sigma$  to the corresponding problem for zeros of the polynomial $P(z)$ lying in the exterior of the circle $|z-O_1|\leq r$. Conditions guarantying such layout of roots \cite{Milovanovic2000,Milovanovic2000a} are obtained, as a rule, from equivalent conditions for the interior of the circle. That is why most of the related results are formulated for the circle with centre at origin. Some of them in addition operate with the unit circle only. To accommodate this observation we introduce two alternative forms of \eqref{zerosPol}:
with the given circle transformed to the unit circle centered at the origin
		\begin{equation}
			P_1(z')= P(O_1+rz')=\suml_{k=0}^{c_n} \alpha'_k z'^k,
		\label{zerosPol1}
		\end{equation}  
and with the given circle transformed to the circle centered at the origin		
		\begin{equation}
			P_2(z'')=P(O_1+z'')=\suml_{k=0}^{c_n} \alpha''_k z''^k.
		\label{zerosPol2}
		\end{equation}		
At this point we need to make use of several results estimating the radius of zero-free circle in terms of the polynomial coefficients. First of these results is a so-called Schur–Cohn test \cite{Henrici1974}. It establishes the necessary and sufficient conditions for the roots of polynomial to lie in the region $|z|>1$. 
\begin{definition}
Given $P^\star(z)=z^n\overline{P(1/\overline{z})}$ we define  the Schur transform $T$ of polynomial $P(z)$ by
$$
\begin{array}{rl}
TP(z):=&\overline{\alpha_0}P(z)-\alpha_n P^\star(z)\\
=& \suml_{k=0}^{n-1} (\overline{\alpha_0} \alpha_k - \alpha_n \overline{\alpha_{n-k}})z^k.
\end{array}
$$
\end{definition}
\begin{theorem}[{\cite{Henrici1974}}]\label{tShur}
Let $P(z)$ is a polynomial of degree $n>0$. All zeros of  $P(z)$ lie in the exterior of the circle $|z| \leq 1$ if and only if for all $k=1,2, \ldots, n$
\begin{equation}
	\gamma_k > 0, 
	\label{shurCrit}
\end{equation}
where $\gamma_k := T^k P(0)$ and $T^k P= T(T^{k-1} P)$. 
\end{theorem}

\begin{corollary}
If the coefficients  $\alpha_k$ satisfy
$$
\alpha_0\geq \alpha_{1}\geq \ldots \geq \alpha_{n-1} \geq \alpha_n > 0,
$$
then all zeros of $P(z)$ lie outside the circle {$|z|\leq 1$}.
\end{corollary}

The application Schur-Cohn test to $P_1(z)$ produces a system of $c_n$-th innequalities for the coefficients of nonlocal condition \eqref{bp1:nc} which is sufficient for the solution's existence. 
Next four Lemmas are more convenient than Theorem \ref{tShur} from the computational standpoint since the number of the produced inequalities are independent on the polynomial degree.    
{
\begin{lemma}\label{lem_zpol_1}
All zeros of $P(z)$  lie in th region 
$$
|z|\geq \frac{|\alpha_0|}{|\alpha_0|+M},
$$
where $M=\max\limits_{1\leq k \leq n}|\alpha_k|$.\\
\end{lemma}
\begin{lemma}\label{lem_zpol_2}
All zeros of $P(z)$ satisfy the inequality 
$$
|z|\geq \frac{|\alpha_0|}{\left[|\alpha_0|+M^q \right]^{1/q}},\quad M=\left(\sum\limits_{k=1}^n |\alpha_k|^p\right)^{1/p},\quad 
p,\ q \in \mathbb{R}_+,\quad 
\frac{1}{p}+\frac{1}{q}=1
$$
\end{lemma}
}{
Next estimate is due to M. Fujiwara \cite{Fujiwara1916}. 
It is an optimal homogeneous estimate  in the space of polynomials \cite{Batra2004}:
\begin{lemma}\label{lem_zpol_3}
All zeros of  $P(z)$ belong to the region 
$$
|z|\geq \frac{1}{2}\min\limits_{\alpha_i\neq 0 }\left\{\left|\frac{\alpha_0}{\alpha_1}\right|, \left|\frac{\alpha_0}{\alpha_2}\right|^{1/2}, \ldots, \left|\frac{\alpha_0}{\alpha_{n-1}}\right|^{1/(n-1)}, \left|\frac{2\alpha_0}{\alpha_n}\right|^{1/n}\right\}. 
$$
\end{lemma}
}  
The last of the estimates given here was proved by H.~Linden.
This estimate in its original form gives bounds on the real and imaginary part of zeros separately. It has been adapted to fit within the framework studied here. 
\begin{lemma}\label{lem_zpol_4}
All  zeros of $P(z)$ belong to the region  $|z|\geq \max\{V_1^{-1},V_2^{-1}\}$, where
$$
V_1=\cos{\frac{\pi}{n+1}}+\frac{|\alpha_n|}{2|\alpha_0|}\left(\left|\frac{\alpha_1}{\alpha_n}\right|+\sqrt{1+\suml_{k=1}^{n-1}\left|\frac{\alpha_k}{\alpha_n}\right|^2}\right),
$$
$$
V_2=\frac{1}{2}\left(\left|\frac{\alpha_1}{\alpha_0}\right|+\cos\frac{\pi}{n}\right)+
\frac{1}{2}\left[\left(\left|\frac{\alpha_1}{\alpha_0}\right|-\cos\frac{\pi}{n}\right)^2+
\left(1+\left|\frac{\alpha_n}{\alpha_0}\right|\sqrt{1+\suml_{k=2}^{n-1}\left|\frac{\alpha_k}{\alpha_n}\right|^2}\right)^2\right]^{1/2}.
$$
\end{lemma}

%

By combining the estimates given by Theorem \ref{tShur} or Lemmas \ref{lem_zpol_1} -- \ref{lem_zpol_4} with Theorem \ref{thmNCExist} we obtain new sufficient conditions for the existence and uniquenesses of the solution to \eqref{bp1}. 

\begin{theorem}\label{thmSufCond}
Assume that  $A$, $f(t)$ and $u_0$ satisfy the conditions of Theorem \ref{thmNCExist}. 
The generalized solution \eqref{bp1IntRed} of nonlocal problem \eqref{bp1} exists if either of the following is true.
\begin{enumerate}
	\item $\exists\ z > 1$, which along with the coefficients of  $P(\varphi(\rho) z)$ from (\ref{zerosPol}) satisfies the conditions of Theorem \ref{tShur} or at least one of Lemmas \ref{lem_zpol_1} -- \ref{lem_zpol_4}.
	\item $\exists\ z > 1$, which along with the coefficients of $P_1(z)$ from (\ref{zerosPol1}) satisfies the conditions of Theorem \ref{tShur} or at least one of Lemmas \ref{lem_zpol_1} -- \ref{lem_zpol_4}.
	\item $\exists\ z > r$, which along with the coefficients of $P_2(z)$ from (\ref{zerosPol2}) satisfies the conditions of at least one of Lemmas \ref{lem_zpol_1} -- \ref{lem_zpol_4}.
\end{enumerate}
\end{theorem}

\begin{proof}
The results formulated in Theorem \ref{tShur} and Lemmas \ref{lem_zpol_1} -- \ref{lem_zpol_4} guaranty that as soon as at least one of three theorem's prepositions is valid all zeros of $B(z)$  will lie outside $\Phi$. Hence according to Theorem \ref{thmNCExist} the solution of \eqref{bp1} exists and is unique.
\end{proof}
Propositions 1-3 of Theorem  \ref{thmSufCond} are ordered in such a way that the first proposition deals with a circle obtained by putting $\theta = \pi/2$. It is therefore valid for $A$ with any $\theta \in (0;\pi/2]$. 
The other two propositions use the parameters of encompassing circle defined above. These propositions will lead to more general sufficient conditions for $\theta <\pi/2$.  
To illustrate these facts we apply different proposition stated in Theorem \ref{thmSufCond} to some concrete examples of nonlocal conditions.
%
%
 
\begin{example}
Let us again consider nonlocal problem (\ref{bp1}) with operator coefficient $A$ ($\theta=\theta_0, \rho=0$) and the Bicadze-Samarskii--type nonlocal condition 
\begin{equation}\label{eq:nc_ex2}
u(0)+\alpha_1 u(t_1)=\alpha_2 u(t_2).
\end{equation}
As we have mentioned in Example \ref{exBS} the set $\Ker(B(z))$ can not be found in a closed form. 

Estimate (\ref{estLiang2002}) yields:
$$
|\alpha_1|+|\alpha_2|<1. 
$$   
%
Meanwhile the application of proposition 1 from Theorem \ref{thmSufCond} together with Schur-Cohn algorithm with $\theta_0=\pi/2$  lead us to system of inequalities:
\begin{equation}\label{eqintshura2}
\left\{
\begin{array}{l}
	|\alpha_2|<1,\\
	|1-\alpha_2^2|>|\alpha_1(1-\alpha_2)|.
\end{array}\right.
\end{equation}
Here we assumed that $t_1=1,\ t_2=2$ and $Q=1$.
Solution of \eqref{eqintshura2} are graphically compared to with set of pairs $(\alpha_1, \alpha_2)$ satisfying \eqref{estLiang2002} in Figure \ref{picShura2} a).
\begin{figure}[ht]%
\begin{tabular}{c}
		\begin{overpic}[width=0.49\linewidth]%
		{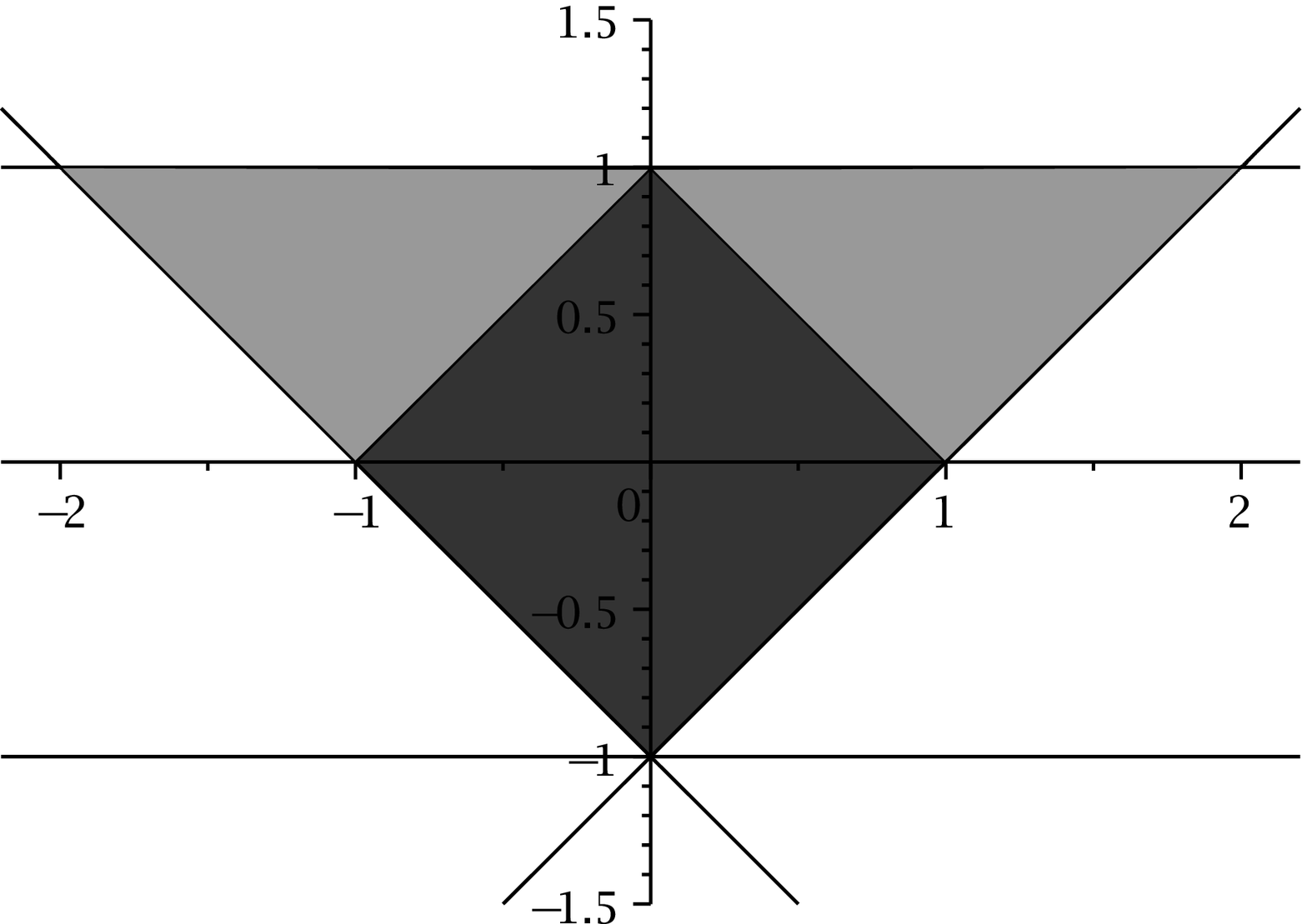}
			\put(5,65){\textbf{a)}}
			\put(95,38){\small $\alpha_1$}
			\put(51,65){\small $\alpha_2$}
		\end{overpic}
		\begin{overpic}[width=0.49\linewidth]%
		{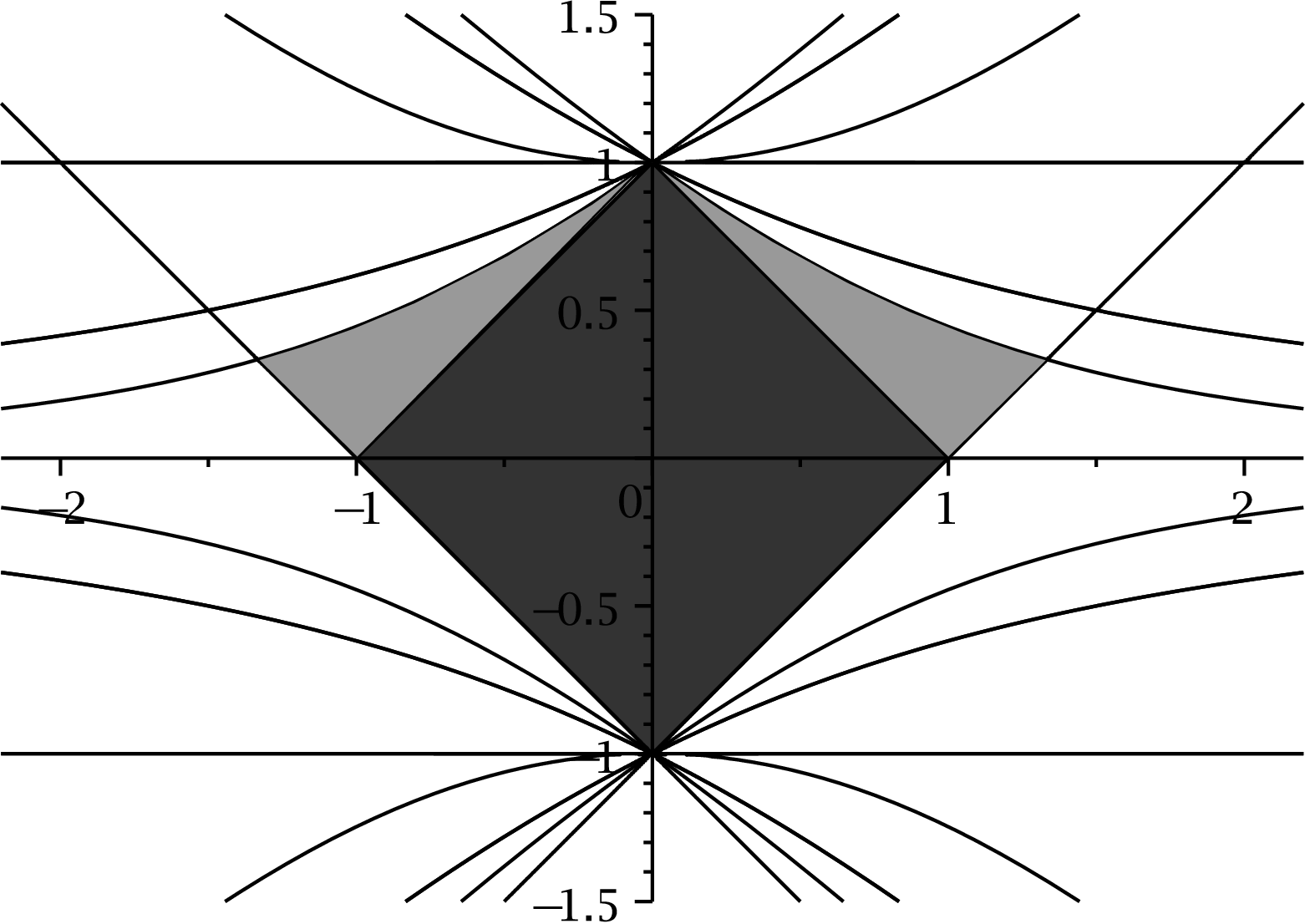}
			\put(5,65){\textbf{b)}}
			\put(95,38){\small $\alpha_1$}
			\put(51,65){\small $\alpha_2$}
		\end{overpic}
				\\
\end{tabular}
\caption{The admissible sets of $\alpha_{1,2} \in \R$ from \eqref{eq:nc_ex2} using condition \eqref{estLiang2002} (coloured in \cbox{gray80}) and  proposition 1 from Theorem \ref{thmSufCond} $\rho=0$, $\forall \theta_0 \in [0, \pi/2)$ (coloured in \cbox{gray40}), where:
\textbf{a)} $t_1=1, t_2=2$; 
\textbf{b)} $t_1=3, t_2=8$.}
\label{picShura2}
\end{figure}
Our approach apparently gives more general conditions than \eqref{estLiang2002} or its particular case from \cite{Byszewski1992} even though we made sufficient conditions \eqref{eqintshura2} independent on $\theta_0$.
To compare different conditions quantitatively we can weight the areas of corresponding admissible parameters sets against each other. 
In that sense the area of admissible parameters set obtained by application of Proposition 1, Theorem \ref{thmSufCond} for $t_1=1,\ t_2=2$ and $Q=1$ is twice bigger than the set obtained by the application of \eqref{estLiang2002}.
The number of inequalities in \eqref{eqintshura2} obtained from \eqref{shurCrit} depends on the ratio $t_2/t_1$ from the nonlocal condition. So it will grow if we increase $t_2/t_1$. 
As a result the corresponding admissible parameters set defined by \eqref{shurCrit} is going to shrink in size and in the limit $t_2/t_1\rightarrow\infty$ will become equal to the set defined by \eqref{estLiang2002}.
To illustrate this behaviour we provide (see Figure \ref{picShura2} b) the same comparison of admissible parameter sets as in Figure \ref{picShura2} a) but for the case $t_1=3,t_2=8$. Recall that all these results are valid  for any sectorial operator $A$ with $\rho=0$.

Proposition 2 of Theorem \ref{thmSufCond} ought to be more advantageous for operator coefficients with some fixed $\theta<\pi/2$. Let us fix $\theta_0 =\pi/3$ and calculate the centre and the radius of circumcircle.
We get 
$$
O_1 = 0.3950734246 , \quad 
r = 0.6049265754.
$$
By substituting  these parameters into \eqref{zerosPol1} and \eqref{zerosPol2} we produce two alternative forms of $P(z)$
$$
\begin{array}{ll}
P_1(z') &\approx 0.37 \alpha_2 {z'}^2+(0.6 \alpha_1+0.48 \alpha_2) z'+1+0.4 \alpha_1+0.16 \alpha_2 , \\
P_2(z'') &\approx \alpha_2 {z''}^2+(\alpha_1+0.79 \alpha_2) {z''}+1+0.4 \alpha_1+0.16 \alpha_2
\end{array}
$$
for the given nonlocal condition.

Application of proposition 2 of Theorem \ref{thmSufCond} along with \eqref{shurCrit} (setting $t_1=1,\ t_2=2$ as before) gives us the set of admissible $(\alpha_1,\alpha_2)$ depicted in Figure \ref{picShura2Pzp1} a). 
\begin{figure}[ht]%
\begin{tabular}{lr}
		\begin{overpic}[width=0.47\linewidth]%
		{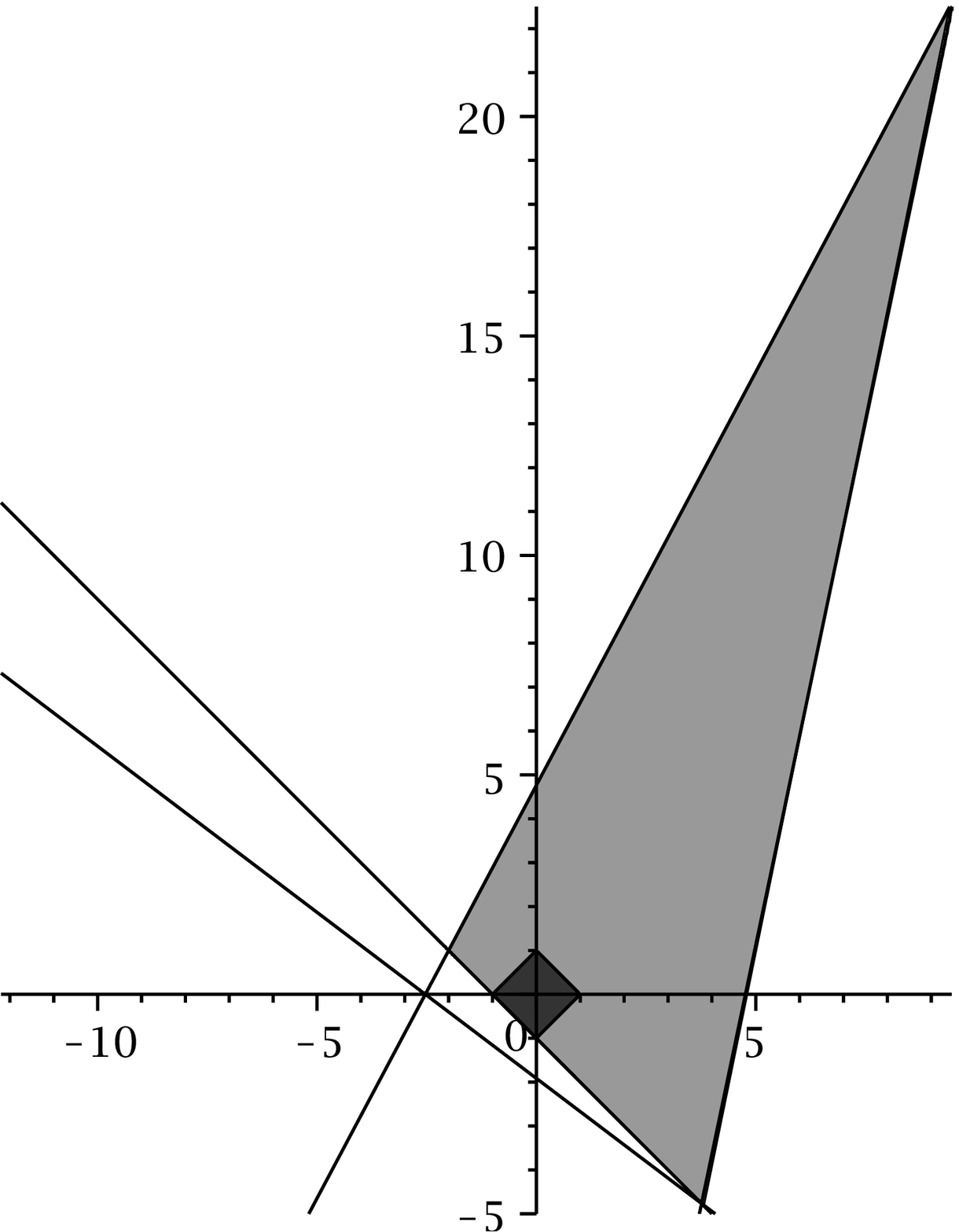}
			\put(5,95){\textbf{a)}}
			\put(73,22){\small $\alpha_1$}
			\put(45,95){\small $\alpha_2$}
		\end{overpic}&
		\begin{overpic}[width=0.49\linewidth]%
		{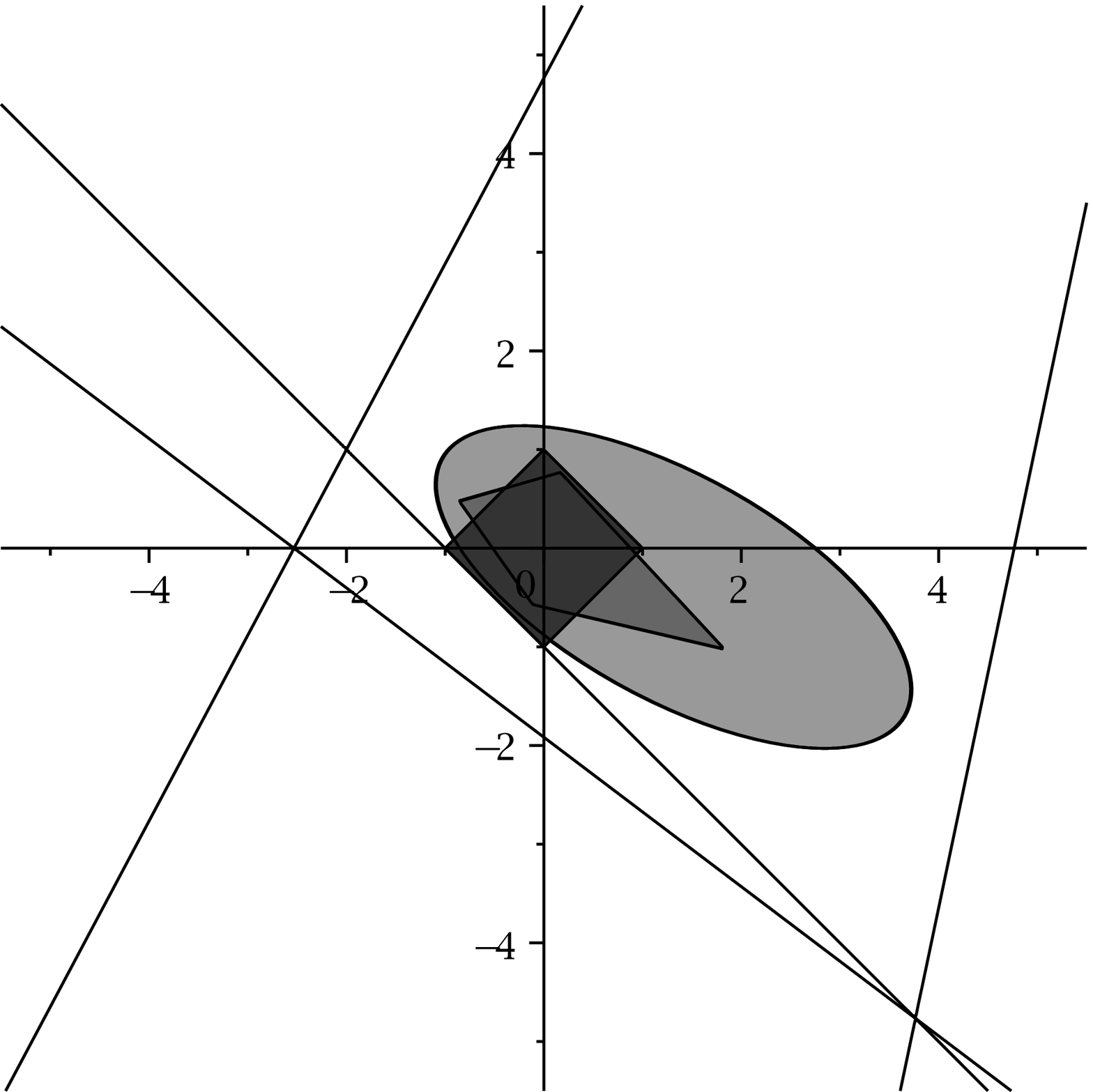}
			\put(5,95){\textbf{b)}}
			\put(95,45){\small $\alpha_1$}
			\put(52,95){\small $\alpha_2$}
		\end{overpic}
				\\
\end{tabular}
\caption{
Admissible values of parameters $\alpha_{1,2} \in \R$ from nonlocal condition \eqref{eq:nc_ex2}  $\rho=0$, $\theta = \pi/3$ obtained via estimate \eqref{estLiang2002} (coloured in \cbox{gray80}) and:
\textbf{a)}  proposition 2 from Theorem \ref{thmSufCond} and estimate \eqref{shurCrit} coloured in \cbox{gray40};
\textbf{b)} proposition 3 of Theorem \ref{thmSufCond} based on Lemma \ref{lem_zpol_1} (coloured in \cbox{gray60}) and Lemma \ref{lem_zpol_2} (coloured in \cbox{gray30}).
}
\label{picShura2Pzp1}
\end{figure}
\FloatBarrier
One can see that this set contains both admissible parameters sets obtained from proposition 1 of the same theorem and condition \eqref{estLiang2002}. 
The area of the admissible parameters set resulting from Schur-Kohn test has grown considerably comparing to Figure \ref{picShura2}. 
Similarly to Example \ref{ex1} this grow is caused by the usage of smaller $\theta_0=\pi/3$, because smaller $\theta_0$ leads to the region $\Phi$ with a smaller diameter. 

All applications of Theorem \ref{thmSufCond} demonstrated hitherto are sensitive to the values $t_i$ from nonlocal condition. For situation where such sensitivity is unfavourable one may wish to use zero bounds from Lemmas \ref{lem_zpol_1} -- \ref{lem_zpol_4} instead of the Schur-Cohn test in propositions of Theorem \ref{thmSufCond}.
Graphical comparisons of the admissible parameters sets given by these Lemmas are presented on figures \ref{picShura2Pzp1} b) trough \ref{picShura2Pzp2} b). 
To make graphical comparisons more straightforward we kept the boundary of the set obtained by application of proposition 2 from Theorem \ref{thmSufCond} (see figure \ref{picShura2Pzp1} a) on all mentioned figures. 
\begin{figure}[ht]%
\begin{tabular}{lr}
		\begin{overpic}[width=0.49\linewidth]%
		{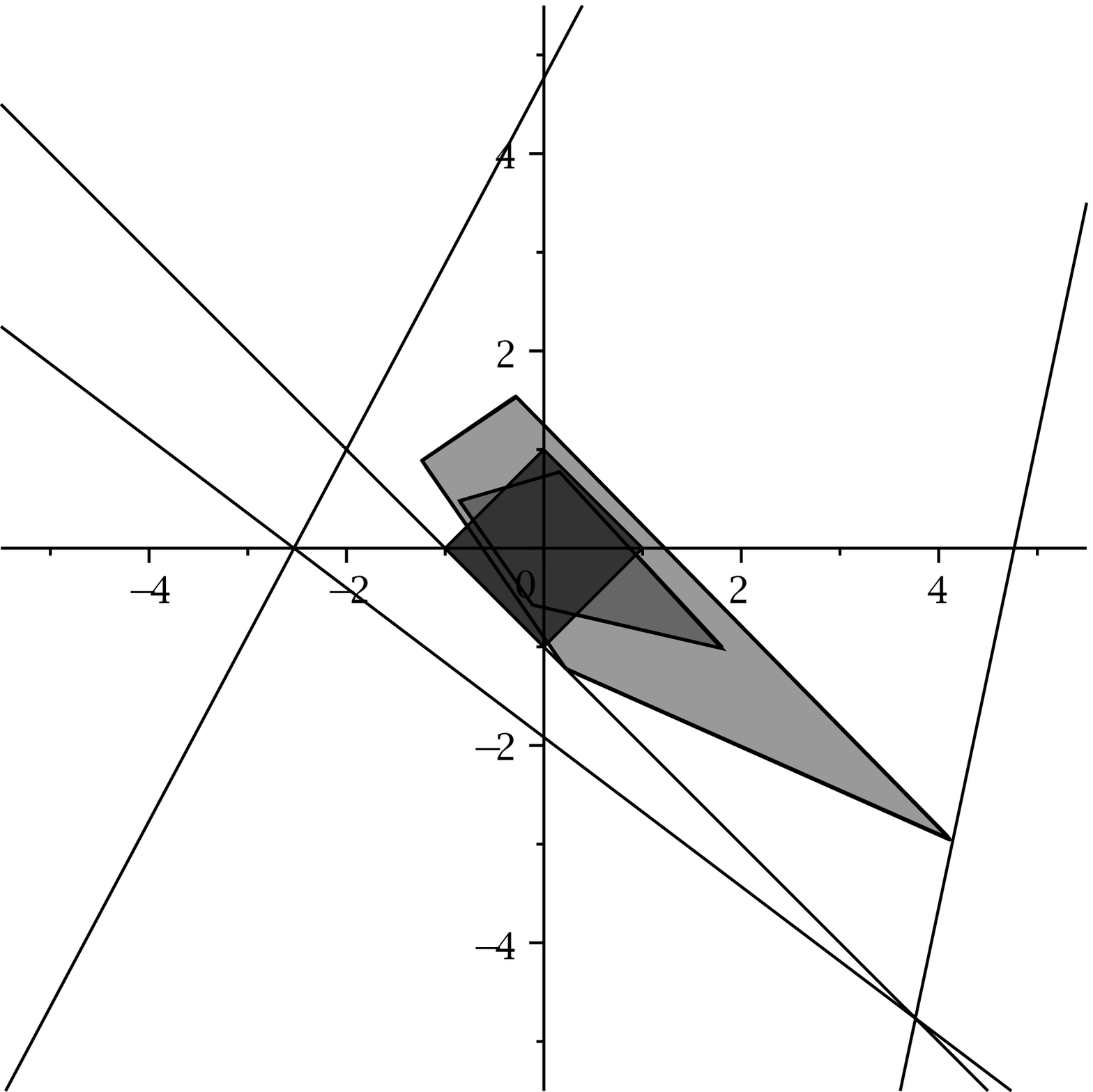}
			\put(5,95){\textbf{a)}}
			\put(95,45){\small $\alpha_1$}
			\put(52,95){\small $\alpha_2$}
		\end{overpic}&
		\begin{overpic}[width=0.49\linewidth]%
		{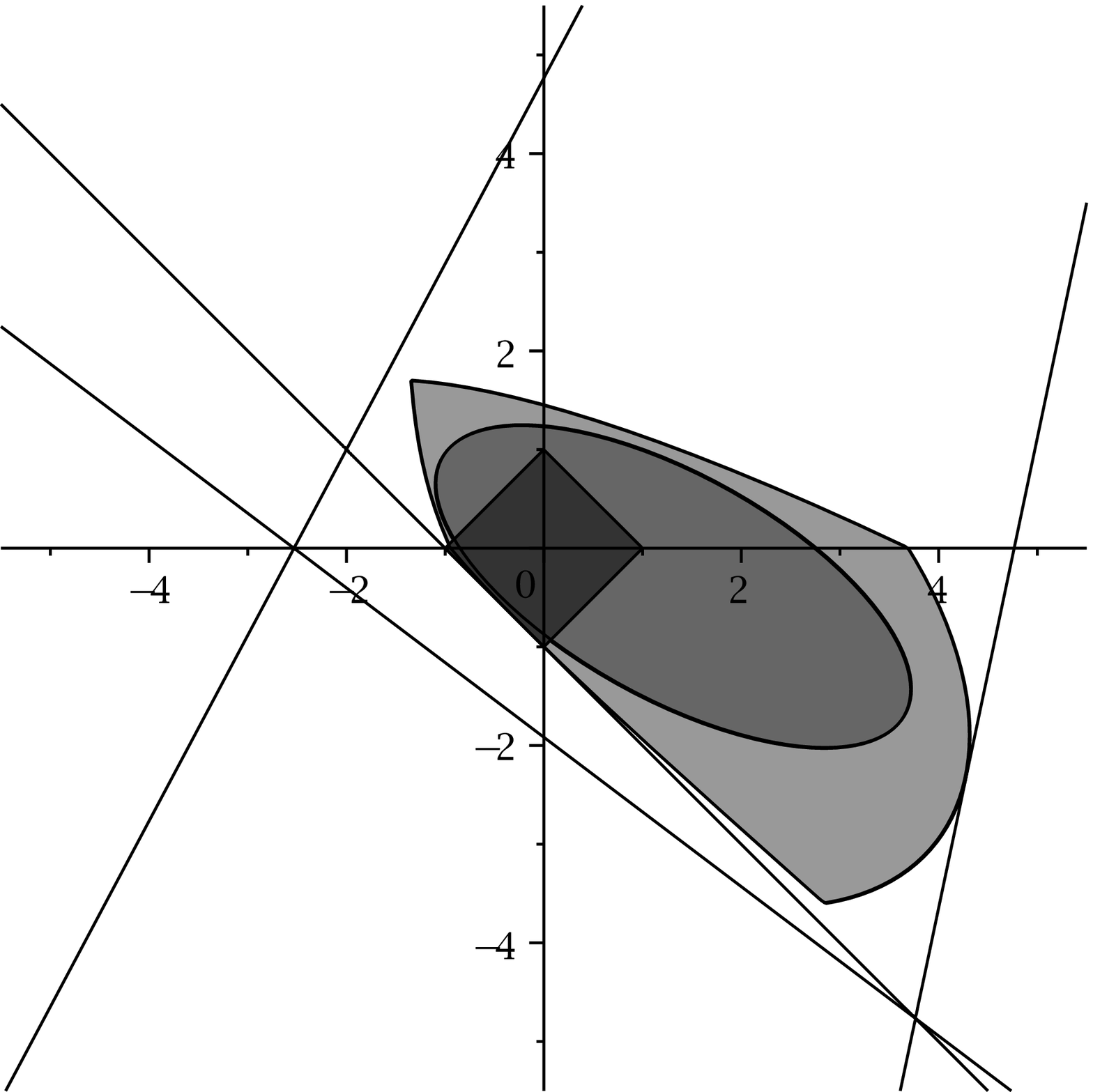}
			\put(5,95){\textbf{b)}}
			\put(95,45){\small $\alpha_1$}
			\put(52,95){\small $\alpha_2$}
		\end{overpic}
				\\
\end{tabular}
\caption{
Admissible values of parameters $\alpha_{1,2} \in \R$ from nonlocal condition \eqref{eq:nc_ex2} $\rho=0, \theta=\pi/3$ obtained via estimate \eqref{estLiang2002} (coloured in \cbox{gray80}) and proposition 3 of Theorem \ref{thmSufCond} based on: 
\textbf{a)} the estimate from Lemma \ref{lem_zpol_1} coloured in \cbox{gray60} and the estimate from Lemma \ref{lem_zpol_3} coloured in \cbox{gray40}
\textbf{b)} the estimate from Lemma \ref{lem_zpol_2} coloured in \cbox{gray60} and the estimate from Lemma \ref{lem_zpol_4} coloured in \cbox{gray40} 
}
\label{picShura2Pzp2}
\end{figure}

One thing the reader would immediately note is that the dominance of necessary conditions presented in this work over the estimate \eqref{estLiang2002} is no longer absolute.  
The sets obtained from Lemmas \ref{lem_zpol_1} -- \ref{lem_zpol_3} do not fully cover the set given by \eqref{estLiang2002}. But still any of these Lemmas performs better than condition \eqref{estLiang2002} in terms of the area in the parameters space.
The same is true for Lemma \ref{lem_zpol_4} which lead to the set having the area at least six times bigger than the area coloured in \cbox{gray80}.  In addition to that the condition given by Lemma \ref{lem_zpol_4} clearly generalize \eqref{estLiang2002}. 
This dominance of Lemma \ref{lem_zpol_4} over the other necessary conditions may not always be valid because the mutual  interdependence of root estimates from  lemmas \ref{lem_zpol_1} -- \ref{lem_zpol_4} are not yet known. 
Therefore, they are advised to be used jointly. 
 
\end{example}
\FloatBarrier
A priori estimates for the zero-free circle mentioned in Lemmas \ref{lem_zpol_1} -- \ref{lem_zpol_4} have been derived by inversion of the corresponding polynomial zero bounds. These particular zero bounds have been selected among others from \cite{Milovanovic2000} \cite{Milovanovic2000a}  based on the numerical comparison of its performance for a number of nonlocal conditions. 
More detailed discussion about various root finding methods and their potential in application to theorems \ref{thmNCExist} and \ref{thmSufCond} are given in \cite{PhdThSytnykEN}.
In this work it is also described how to generalize the results of section \ref{sec3} to the case when some of $t_i$ are irrational. All codes for the generation of figures, circle parameters calculation and numerical checks of the presented necessary conditions can be found at\newline \url{imath.kiev.ua/~sytnik/research/works/nonlocal-2014}.

\section{Conclusions and future work}
By exploiting the connection between nonlocal evolution problem \eqref{bp1} and its classical counterpart we derive the reduction operator representation \eqref{reprDelta}. 
It enabled us to work out the conditions for existence of the mild solution to \eqref{bp1}.
Analogous existence analysis is possible for other evolution problems  as long as the exact representation of $B(z)$ similar to \eqref{zerosExp} is obtainable and one can characterize the set free from zeros of $B(z)$.  The nonlocal evolutional problems for the abstract time dependant Schro\"dinder equation and the abstract second order linear differential equation are both tractable by our approach. They will constitute the subject for our future analysis.  
\selectlanguage{english}

\end{document}